\documentclass[reqno]{amsart}

\usepackage[english]{babel}
\usepackage[all]{xy}
\usepackage{amssymb}
\usepackage{amsmath}
\usepackage{amsfonts}
\usepackage{url}

\newtheorem{theorem}{Theorem}[section]
\newtheorem{prop}[theorem]{Proposition}
\newtheorem{lemma}[theorem]{Lemma}

\theoremstyle{definition}

\newtheorem{rem}[theorem]{Remark}

\newcommand{\lra}{\longrightarrow}

\newcommand{\ie}{{\it i.e.}\ }

\newcommand{\IC}{\mathbb{C}}

\newcommand{\IG}{\mathbb{G}}
\newcommand{\IN}{\mathbb{N}}
\newcommand{\IP}{\mathbb{P}}
\newcommand{\IQ}{\mathbb{Q}}
\newcommand{\IR}{\mathbb{R}}
\newcommand{\IZ}{\mathbb{Z}}

\newcommand{\cA}{\mathcal{A}}
\newcommand{\cD}{\mathcal{D}}

\newcommand{\cO}{\mathcal{O}}

\newcommand{\coloneqq}{:=}

\DeclareMathOperator{\rank}{rank}
\newcommand{\Aut}{{\rm Aut}}
\newcommand{\Dih}{{\rm Dih}}
\newcommand{\id}{{\rm id}}
\newcommand{\NS}{{\rm NS}}
\newcommand{\Trans}{{\rm Trans}}
\newcommand{\Ker}{{\rm Ker}}
\newcommand{\Pic}{{\rm Pic}}
\newcommand{\Grass}{{\rm Grass}}

\newcommand{\set}[1]{\left\{ #1 \right\}}
\newcommand{\pa}[1]{\left( #1 \right)}

\begin{document}

\date{\today}

\title[]{The automorphism group of the Hilbert scheme of two points on a generic projective K3 surface}
\author[Boissi\`ere, Cattaneo, Nieper-Wisskirchen, Sarti]{Samuel Boissi\`ere, Andrea Cattaneo, Marc Nieper-Wisskirchen and Alessandra Sarti}

\address{Samuel Boissi\`ere, Universit\'e de Poitiers, 
Laboratoire de Math\'ematiques et Applications, 
 T\'el\'eport 2 
Boulevard Marie et Pierre Curie
 BP 30179,
86962 Futuroscope Chasseneuil Cedex, France}
\email{samuel.boissiere@math.univ-poitiers.fr}
\urladdr{http://www-math.sp2mi.univ-poitiers.fr/$\sim$sboissie/}

\address{Andrea Cattaneo, Dipartimento di Matematica, Universit\`a di Parma, Parco Area delle Scienze 53/A, 43124, Parma, Italy} 
\email{andrea.cattaneo@unipr.it}

\address{Marc Nieper-Wi{\ss}kirchen, Lehrstuhl f\"ur Algebra und Zahlentheorie, Universit\"ats\-stra{\ss}e~14, D-86159 Augsburg}
\email{marc.nieper-wisskirchen@math.uni-augsburg.de}
\urladdr{http://www.math.uni-augsburg.de/alg/}

\address{Alessandra Sarti, Universit\'e de Poitiers, 
Laboratoire de Math\'ematiques et Applications, 
 T\'el\'eport 2 
Boulevard Marie et Pierre Curie
 BP 30179,
86962 Futuroscope Chasseneuil Cedex, France}
\email{sarti@math.univ-poitiers.fr}
\urladdr{http://www-math.sp2mi.univ-poitiers.fr/$\sim$sarti/}


\begin{abstract}
 We determine the automorphism group of the Hilbert scheme of two points on a generic projective K3 surface of any polarization.
We obtain in particular new examples of Hilbert schemes of points having non-natural non-symplectic automorphisms. The existence
of these automorphisms depends on solutions of Pell's equation.
\end{abstract}

\maketitle

\section{Introduction}

A classical result in the theory of surfaces is that any complex K3 surface $S$ which contains an ample divisor 
$D$ with $D^2=2$ is a double cover of the plane ramified over a smooth sextic curve (see~\cite{saintdonat}); in particular,
the covering involution is an anti-symplectic automorphism whose induced action on $H^2(S,\IZ)$ is the reflection
in the span of~$D$. O'Grady~\cite{ogrady1,ogrady2} has given conjectural generalizations
of this statement to  higher dimensional holomorphic symplectic manifolds $X$ which are deformations of the Hilbert scheme of $n$ points~$S^{[n]}$ on a K3 surface $S$ and which are polarized by an ample divisor $D$ of square $2$
with respect to the Beauville--Bogomolov quadratic form on $H^2(X,\IZ)$. There is a moduli space parametrizing
degree $2$ polarized irreducible holomorphic symplectic manifolds ($X,D)$ with $X$ deformation of~$S^{[n]}$. 
The ``L Conjecture'' of O'Grady~\cite{ogrady1} states that there is an open dense subset of this moduli space
which parametrizes pairs $(X,D)$ such that the linear system~$|D|$ is base-point-free and induces a regular
map $X\to |D|^\ast$ which is of degree $2$ onto its image $Y$. In particular, the covering involution is
non-symplectic and its action on $H^2(X,\IZ)$ is the reflection in the span of $D$. This conjecture is 
particularly interesting in the case $n=2$ where O'Grady~\cite{ogrady2} proves that up to deformation 
there are two possibilities: either $X$ is a double cover of an EPW sextic or $X$ is birational to a hypersurface
of degree at most $12$. It is conjectured that the second case can not happen.

The non-symplectic involutions on deformations of $S^{[2]}$ have been classified by Beauville~\cite{BeauvilleInvol}
by means of some numerical invariants of the fixed surface and by Ohashi--Wandel~\cite{OW}, Boissi\`ere--Camere--Sarti~\cite{BoissiereCamereSarti}
and Mongardi--Wandel~\cite{MW} by means of the properties of the invariant lattice and its orthogonal complement. New examples of non-symplectic involutions
on deformations of $S^{[2]}$ have thus been obtained, but not on $S^{[2]}$ itself. 

In this paper, we answer the original question without deformation: what are the automorphisms of $S^{[2]}$ itself? 
We study the generic case where $S^{[2]}$ has Picard number $2$, which is the mimimal possible rank. The surface $S$
is a generic algebraic K3 surface of Picard number one, its N\'eron--Severi group is 
generated by an ample divisor $H$ of self-intersection $H^2=2t$ with $t\geq 1$. If $t=1$ then $S$ is the double cover
of $\IP^2$ branched along a smooth sextic curve and we show in Proposition~\ref{prop: nonsympl inv} that in fact $\Aut(S^{[2]})$ is isomorphic to $\IZ/2\IZ$ and it is generated by the non-symplectic involution on $S^{[2]}$ induced by the covering involution on $S$. 
The main result of the paper (see Section~\ref{ss:main})  gives a complete description 
of the group of automorphisms of $S^{[2]}$ when $t\geq 2$:

\begin{theorem} Let $S$ be an algebraic K3 surface such 
that $\Pic(S)=\IZ H$ with $H^2=2t$, $t\geq 2$. Then $S^{[2]}$ admits a non-trivial automorphism
if and only if one of the following equivalent conditions is satisfied:
\begin{enumerate}
\item $t$ is not a square, Pell's equation $x^2-4ty^2=5$ has no solution and
  Pell's equation $x^2-ty^2=-1$ has a solution;
\item there exists an ample class $D \in \NS(S^{[2]})$ such that $D^2 = 2$.
\end{enumerate}
Moreover, if this is the case the class $D$ is unique, the automorphism is unique and it is a non-symplectic involution
whose action on $H^2(S^{[2]},\IZ)$ is the reflection in the span of $D$.
\end{theorem}

The case $t=2$ corresponds to the situation where $S$ is a generic quartic in $\IP^3$ and the non-symplectic involution
is Beauville's one~\cite{beauville2}. The next cases are $t=10,13$ or~$17$ and our result shows the existence of a non-symplectic involution
on the Hilbert scheme of two points on a generic K3 surface polarized by a class of square $20,26$ or~$34$.

\medskip

{\it Acknowledgements.} We thank Kieran O'Grady, Brendan Hassett and Emanuele Macr\`i for very helpful explanations. The second author was partially supported by 
the Italian-French Research Network Program GDRE-GRIFGA and thanks the hospitality of the
University of Poitiers where most of the work was done.

\section{Preliminary results}

In this paper, $S$ denotes an algebraic complex K3 surface with $\Pic(S)=\IZ H$. Since $H^2>0$, 
$H$ or $-H$ is effective so we can assume that $H$ is effective. By Nakai's criterion $H$ is ample and $S$ is projective.
We have $H^2=2t$ with $t\geq 1$ and~$H$ is very ample if $t\geq 2$ (see \cite[p.623]{saintdonat} or \cite{andreas2}).

\subsection{Basic results on Pell's equation}\label{ss:Pell}

For any $t\in\IN$ that is not a square  and $m\in\IZ$ we consider Pell's equation
$$
P_t(m)\colon x^2-ty^2=m,
$$
for $x,y$ integers. A solution $(x,y)$ of this equation is called \emph{positive}  
if $x>0$, $y>0$ and the positive solution with minimal $x$ is called the \emph{minimal} one.
Consider the real quadratic field $\IQ[\sqrt{t}]$. 
The \emph{norm} of any $z\coloneqq x+y\sqrt{t}\in\IQ[\sqrt{t}]$
is defined by $N(z)\coloneqq x^2-ty^2$. Using the identity $(x+y\sqrt{t})(x-y\sqrt{t})=N(z)$ it is easy to check
that a solution $(x,y)$ of $P_t(\pm 1)$ is positive if and only if $z=x+y\sqrt{t} >1$. It follows that the minimal
solution of $P_t(\pm 1)$ is the minimal \emph{real number} $z\in\IR$ satisfying $z>1$, $z\in\IZ[\sqrt{t}]$ and $N(z)=\pm 1$.

By a theorem of Lagrange, the \emph{continued fraction} expansion
$$
\sqrt{t} = a_0 + \frac{1}{a_1 + \frac{1}{a_2 + \frac{1}{\ddots}}},
$$
has the property that the sequence of positive integers $(a_i)_{i\geq 1}$ is 
periodic~\cite[Theorem~VII.3]{sierpinski}; we denote by $s$ its period. We define the $k$-th \emph{convergent}
of $\sqrt{t}$ as the rational number
$$
C_k = a_0 + \frac{1}{a_1 + \frac{1}{\ddots + \frac{1}{a_k}}}.
$$
We denote by $x_k$ (resp. $y_k$) the numerator (resp. denominator) of $C_k$.

Pell's equation $P_t(1)$ has a solution for any value of $t$. If $s$ is even, the positive solutions are the pairs $(x_{ns-1},y_{ns-1})$ for $n\geq 1$; if $s$ is odd, the positive solutions are the pairs $(x_{2ns-1},y_{2ns-1})$ for $n\geq 1$ \cite[Theorems~VIII.7~\&~VII.8]{sierpinski}.

Pell's equation $P_t(-1)$ has a solution if and only if $s$ is odd, in which case the positive solutions are the pairs $(x_{(2n - 1)s - 1},y_{(2n - 1)s - 1})$ for $n\geq 1$ \cite[Theorem~VIII.9]{sierpinski}.

The following lemma is certainly well-known, we include it for convenience:

\begin{lemma}\label{lem:Pell}
Let $(\alpha,\beta)$ be the minimal solution of the equation $P_t(1)$. If the equation $P_t(-1)$ has a solution, then its minimal solution $(a,b)$ satisfies:
$$
\alpha=2a^2+1, \quad \beta=2ab.
$$
\end{lemma}

\begin{proof}
Put $Z\coloneqq \alpha+\beta\sqrt{t}$ and $z\coloneqq a+b\sqrt{t}$. We have $Z>1$, $z>1$ and
$\frac{z}{Z}<z$. Since $\frac{z}{Z}\in\IZ[\sqrt{t}]$ has norm $-1$, by minimality of $z$ this implies
that $\frac{z}{Z}< 1$, so $1<z< Z$ and $1<z^2<Z^2$ with $N(z^2)=1$. It is easy to see 
that all positive solutions of $P_t(1)$ are of the form $Z^n$ for some $n\geq 1$ (see for instance~\cite[Theorem~II.15]{sierpinski}). Since $1<Z<Z^2<\cdots$ we get $z^2=Z$,
hence $\alpha=2a^2+1, \beta=2ab$.
\end{proof}

\begin{rem}\label{rem:solPell}
With the same notation, putting $Z=A+B\sqrt{t}$, we see that the integer solutions of $P_t(1)$ are $(1,0)$ and $\pm Z^n$ with $z\in\IZ$. Putting $z=a+b\sqrt{t}$,
the integer solutions of $P_t(-1)$ are $\pm z^{2n+1}$ with $n\in\IZ$.
\end{rem}

\subsection{Basic facts on the Hilbert scheme $S^{[2]}$}

 We denote by $S^2$ the product of two copies of $S$ and 
by $p_i\colon S^2\to S$ the projection onto the $i$-th factor, $i=1,2$. Consider the symmetric quotient 
$S^{(2)}\coloneqq S^2/\mathfrak{S}_2$ where the symmetric group $\mathfrak{S}_2$ acts by permutation 
of the variables and denote by $\pi\colon S^2\to S^{(2)}$ the quotient map. The variety $S^{(2)}$ is
singular on the locus $\Delta$ which is the image of the diagonal of $S^2$ by $\pi$.
We denote by $S^{[2]}$ the Hilbert scheme of two points on $S$ which parametrizes the length two zero dimensional
subschemes of $S$. The Hilbert--Chow morphim $\rho\colon S^{[2]}\to S^{(2)}$ is projective and birational,
it is a resolution of the singularities. We denote by $E\coloneqq \rho^{-1}(\Delta)$ the exceptional divisor, which
is irreducible.

Recall that by a result of Beauville--Fujiki~\cite{beauville1} the variety $S^{[2]}$ is hyperk\"ahler and the space $H^2(S^{[2]},\IZ)$ is a lattice for the Beauville--Bogomolov--Fujiki quadratic form, isometric to $H^2(S,\IZ)\oplus \IZ \delta$ where $[E]=2\delta$.
We have $\delta^2=-2$ and thus the lattice $H^2(S^{[2]},\IZ)$ is isometric to $U^{\oplus 3}\oplus E_8^{\oplus 2}\oplus \langle -2\rangle$ 
where $U$ is the unique even unimodular hyperbolic lattice of rank $2$ and $E_8$ is the even negative definite lattice
of rank $8$ associated to the Dynkin diagram $E_8$.

There exists a natural morphism of groups $\Pic(S)\to\Pic(S^{[2]})$, $L\mapsto L_2$, constructed as follows: for any line bundle $L\in\Pic(S)$, the line bundle $p_1^\ast L\otimes p_2^\ast L$ projects to a line bundle 
$\overline{L}$ on $\Pic(S^{(2)})$ with $\pi^\ast\overline{L}\cong p_1^\ast L\otimes p_2^\ast L$ and 
one defines $L_2\coloneqq \rho^\ast\overline{L}$. Denoting by $\Pic(S)_2$ the set of isomorphism classes of line bundles of the form $L_2$ one has
$$
\Pic(S^{[2]})=\Pic(S)_2\otimes \IZ\cD
$$
where $\cD^2\cong\cO(-E)$ and $c_1(\cD)=-\delta$. In particular, putting $h\coloneqq H_2$, in our situation the sequence $(h,-\delta)$ is a basis of the N\'eron--Severi lattice $\NS(S^{[2]})$ 
whose bilinear form is:
$$
\left(
\begin{matrix}
2t & 0\\ 0 & -2
\end{matrix}
\right)
$$

\subsection{Basic results on the group $\Aut(S^{[2]})$}

We denote by $\Aut(S^{[2]})$ the group of biholomorphic automorphisms of $S^{[2]}$, which is a discrete group (see~\cite{boissiere}). Every automorphism $f\in\Aut(S)$ induces an automorphism denoted $f^{[2]}$ on $S^{[2]}$, such automorphisms are called \emph{natural}. Associating to each
 automorphism $\varphi$ of $S^{[2]}$ the isometry $(\varphi^{-1})^\ast$ of $H^2(S^{[2]},\IZ)$ we get
a morphism of groups $\Aut(S^{[2]})\to O(H^2(S^{[2]},\IZ))$; Beauville~\cite[Proposition~10]{beauville2} proved that this morphism injective. 
We consider the morphism obtained by restriction to the N\'eron--Severi group:
$$
\Psi\colon\Aut(S^{[2]})\to O(\NS(S^{[2]})), \quad \varphi\mapsto \left.(\varphi^{-1})^\ast\right|_{\NS(S^{[2]})}.
$$
The following result is well-known (see for instance \cite[Corollary~15.2.12]{huybrechts}):

\begin{lemma}\label{lem:AutS}
 Let $S$ be an algebraic K3 surface such that $\Pic(S)=\IZ H$, $H^2=2t$, $t\geq 1$. 
\begin{enumerate}
\item\label{lem:AutS_item1} If $t\geq 2$ then $\Aut(S)=\{\id_S\}$.
\item\label{lem:AutS_item2} If $t=1$ then $S$ is the double cover of $\IP^2$ branched along a smooth sextic curve and $\Aut(S)=\{\id_S,\iota\}$ where $\iota$ is the covering involution.
\end{enumerate}
\end{lemma}

\begin{lemma}\label{lem:AutHilbS} Let $S$ be an algebraic K3 surface such that $\Pic(S)=\IZ H$, $H^2=2t$, $t\geq 2$. Then $\Ker(\Psi)\cong \Aut(S)$.
In particular if $t\geq 2$  the morphism $\Psi$
 is injective.
\end{lemma}

\begin{proof}
If $\left.(\varphi^{-1})^\ast\right|_{\NS(S^{[2]})}$ is the identity, then in particular it leaves invariant 
the class $\delta$. By Boissi\`ere--Sarti~\cite[Theorem~1]{boissieresarti} this implies that $\varphi$ is 
a natural automorphism: $\varphi=f^{[2]}$ for some $f\in\Aut(S)$. By Lemma~\ref{lem:AutS}, if $t\geq 2$ one 
has $f=\id_S$ so $\Psi$ is injective.  
\end{proof}

\section{The ample cone of $S^{[2]}$}

In  this section we determine the ample cone $\cA_{S^{[2]}}\subset \NS(S^{[2]})$ in the basis $(h,-\delta)$.
We first recall a classical method due to Beltrametti--Sommese~\cite{belsom} and Catanese--G{\"o}ttsche~\cite{catgo} to construct ample classes and then we give a full description of the ample
cone using recent results of Bayer--Macr\`i~\cite{bayermacri}. Both points of views will be needed in the sequel.
Earlier related results were obtained by Hassett--Tschinkel~\cite{HT,HT2} and  Markman~\cite{Markman2}.

\subsection{The map to the Grassmannian}\label{ss:grass}

Let $L=aH$, $a>0$, be an ample line bundle on $S$ and consider the Grassmannian  $\IG := \Grass\pa{2, H^0(S, L)^\ast}$
of $2$-dimensional subspaces of $H^0(S, L)^\ast$. If $Z\subset S$ is any $0$-cycle, the exact sequence
$$
0 \lra L \otimes I_Z \lra L \lra L \otimes \cO_Z \lra 0
$$
induces an exact cohomology sequence:
$$
0 \lra H^0(S, L \otimes I_Z) \lra H^0(S, L)\overset{r_Z}{\lra} H^0(S, L \otimes \cO_Z) \lra \cdots
$$
Following \cite{belsom,catgo} the line bundle $L$ is called $2$-very ample if the restriction map $r_Z$ is
onto for any $0$-cycle $Z$ of length less than or equal to $3$. If $L$ is 
very ample it defines a morphism
$$
\phi: S^{[2]} \longrightarrow \IG, \quad [Z] \mapsto H^0(S, L\otimes I_Z).
$$
By Catanese--G{\"o}ttsche \cite[Main Theorem]{catgo} $\phi$ is an embedding if and only if 
$L$ is $2$-very ample.

\begin{prop}\label{prop:ample_classes} Let $S$ be an algebraic K3 surface such that $\Pic(S)=\IZ H$ with $H^2=2t$, $t\geq 1$.

\begin{enumerate}
 \item If $t\geq 4$ then $ah-\delta$ is ample on $S^{[2]}$ if $a\geq 1$.
\item If $t\in\{2,3\}$ then $ah-\delta$ is ample on $S^{[2]}$ if  $a\geq 2$.
\item If $t=1$ then $ah-\delta$ is ample on $S^{[2]}$ if  $a\geq 3$.
\end{enumerate}
\end{prop}

\begin{proof}
If $L=aH$ is $2$-very ample then $\phi$ is an embedding so $\phi^\ast \cO_\IG(1)$ is ample on~$S^{[2]}$. As explained 
in \cite[Section~2]{bertramcoskun} its first Chern class is equal to $c_1(L_2)-\delta=ah-\delta$,
so the class $ah-\delta\in \NS(S^{[2]})$ is ample.
 
By Knutsen~\cite[Theorem 1.1]{andreas}, $L$ is $2$-very ample if and only if $L^2\geq 8$ and there exists
no effective divisor $D$ satisfying the following conditions:
\begin{enumerate}
\item $2 D^2 \overset{\text{(i)}}{\leq} L \cdot D \overset{\text{(ii)}}{\leq} D^2 +3 \overset{\text{(iii)}}{\leq} 6$;
\item condition (i) is an equality if and only if  $L \sim 2D$ and $L^2 \leq 12$;
\item condition (iii) is an equality if and only if $L \sim 2D$ and  $L^2 = 12$.
\end{enumerate}
Since $L^2=2ta^2$ we get immediately
$$
L^2\geq 8 \Leftrightarrow
\begin{cases}
a \geq 2 & \text{if } t = 1, 2, 3\\
a \geq 1 & \text{if } t \geq 4.
\end{cases}
$$
Let $D=nH$ be an effective divisor ($n >0$) satisfying condition (iii). We get $2tn^2\leq 3$ so this case
happens only when $t=1$, with $D=H$. As a consequence, the line bundle $L=aH$ is $2$-very ample
for any $a\geq 2$ if $t=2,3$ and for any $a\geq 1$ if $t\geq 4$. Assume now that $t=1$ and $D=H$ satisfies condition (ii).
We get $a\in\{1,2\}$. Condition (i) is not satisfied if $a=1$ but all conditions are satisfied if $a=2$. 
Hence if $t=1$ the line bundle $L=aH$ is $2$-very ample for any $a\geq 3$. 
\end{proof}

We denote by $(x,y)$ the coordinates in $\NS(S^{[2]})\otimes_\IZ\IR$ corresponding to the class $xh-y\delta$.
Observe that: $h$ is a nef and non-ample class;  $-\delta$ is not ample; $3h - \delta$ is ample 
by Proposition \ref{prop:ample_classes}. Hence we 
have $\cA_{S^{[2]}} \subseteq \left\{xh-y\delta\,|\, x>0, y>0 \right\}$.

\subsection{The ample cone and Pell's equation}

Bayer--Macr\`i \cite{bayermacri} use wall-crossing with respect to Bridgeland stability conditions to determine
the movable cone of moduli spaces of sheaves on K3 surfaces. In the particular case of the Hilbert scheme
of two points on a generic K3 surface, a direct application of \cite[Proposition~13.1, Lemma~13.3]{bayermacri} 
gives a full description of the ample cone $\cA_{S^{[2]}}$ depending on solutions of Pell's equation.

\begin{prop}\label{prop:nefcone}
Let $S$ be an algebraic K3 surface such that $\Pic(S)=\IZ H$ with $H^2=2t$, $t\geq 1$.
\begin{enumerate}
\item If $t$ is a square, $t = k^2$  with $k\geq 1$, then $\cA_{S^{[2]}}$ 
is the interior of the cone generated by $h$ and $h - k \delta$.

\item If $t$ is not a square and  Pell's equation $P_{4t}(5)$ has a solution, 
then $\cA_{S^{[2]}}$ is the interior of the cone generated by $h$ and $x h - 2t y \delta$ where
 $(x, y)$ is the minimal solution of $P_{4t}(5)$.

\item If $t$ is not a square and  Pell's equation $P_{4t}(5)$
has no solution, then $\cA_{S^{[2]}}$ is the interior of the cone generated by $h$ and $x h - t y \delta$ where 
$(x, y)$ is the minimal solution of  Pell's equation $P_t(1)$.
\end{enumerate}
\end{prop}

\begin{rem}\label{rem: cone and t}
If $t$ is not a square, the knowledge of the ample cone of $S^{[2]}$ 
determines in which of the cases of Proposition \ref{prop:nefcone} we are, and
in particular whether  Pell's equation $P_{4t}(5)$ admits a solution, since  $P_t(1)$ and $P_{4t}(5)$ have no common solution.
\end{rem}

\section{The isometry group of $\NS(S^{[2]})$}

Recall that in the basis $(h,-\delta)$ of the N\'eron--Severi lattice $\NS(S^{[2]})$ the bilinear form is represented by the matrix
$$
\left(
\begin{matrix}
2t & 0\\ 0 & -2
\end{matrix}
\right).
$$
In this section, we describe the group $O(\NS(S^{[2]}))$ of isometries of the N\'eron--Severi lattice of $S^{[2]}$. The matrix in the basis $(h, -\delta)$ of such an isometry is 
$$
M = \left( \begin{array}{cc}
A & B\\
C & D
\end{array} \right)
$$
and the following conditions hold:
\begin{enumerate}
\item $\det M = \pm 1$, \ie $AD - BC = \pm 1$;
\item $2t = h^2 = (Ah - C \delta)^2$, \ie $C^2 = t(A^2 - 1)$;
\item $-2 = (-\delta)^2 = (Bh - D \delta)^2$, \ie $D^2 = tB^2 + 1$;
\item $0 = -h\delta = (Ah - C\delta)(Bh - D\delta)$, \ie  $CD = t AB$.
\end{enumerate}
We deduce easily that $M$ can be of one of the following two forms:
$$
\left( \begin{array}{cc}
A & B\\
tB & A
\end{array} \right) \text{ or }
\left( \begin{array}{cc}
A & B\\
-tB & -A
\end{array} \right),\quad \text{with }A^2 - t B^2 = 1.
$$
Consider the abelian group
$$
N \coloneqq \left\{ 
\left( 
\begin{matrix}
A & B\\
tB & A
\end{matrix} 
\right) 
\,\Big|\, A, B \in \IZ, A^2 - t B^2 = 1 \right\}\subset O(\NS(S^{[2]}))
$$
and the element $s\coloneqq \left(\begin{matrix} 1 & 0 \\ 0 & -1\end{matrix}\right)\in O(\NS(S^{[2]}))$. 
It is easy to see that $O(\NS(S^{[2]}))$ is the generalized dihedral group of $N$:
$$
O(\NS(S^{[2]}))\cong\Dih(N)\cong N\rtimes \IZ/2\IZ
$$
where $\IZ/2\IZ=\langle s\rangle$ acts by conjugation on $N$.

\begin{rem} \label{rem:t_square}
If $t$ is a square, the solutions of the equation $A^2 - t B^2 = 1$ are $A=\pm 1$, $B = 0$
so $O(\NS(S^{[2]}))$ is isomorphic to the dihedral group with four elements $\set{\id, -\id, s, -s}$.
\end{rem}

\begin{rem}
The isometries of the lattice $\langle 2t\rangle\oplus \langle -2\rangle$ were computed in a different context 
in \cite{gil}, where Bini studies the automorphism group of a K3 surface 
of Picard number two and N\'eron--Severi group isometric to $\langle 2nt\rangle\oplus \langle -2n\rangle$, $n, t$ positive integers. 
\end{rem}

The next proposition shows that the non-trivial isometries of $\NS(S^{[2]})$ induced by automorphisms characterize
the ample cone of $S^{[2]}$ and {\it vice versa}. This very precise link will be the key to the full description
of the automorphism group of~$S^{[2]}$.

\begin{prop}\label{prop:type_of_auts} Let $S$ be an algebraic K3 surface such 
that $\Pic(S)=\IZ H$ with $H^2=2t$, $t\geq 1$ and $f\in\Aut(S^{[2]})$. 
If the isometry on $\NS(S^{[2]})$ induced by $f$ is not the identity then 
it is the involution represented in the basis $(h,-\delta$) by the matrix
$$
\left( \begin{matrix}
A & B\\
-tB & -A
\end{matrix} \right) \quad \text{with } A^2 - t B^2 = 1,\quad A > 0,\quad B < 0
$$
where $A$ and $B$ are uniquely determined by the ample cone of $S^{[2]}$, which is:
$$
\cA_{S^{[2]}} = \left\{ xh-y\delta\,|\,y  > 0, Ay  < -tBx \right\}.
$$
\end{prop}

\begin{proof}
Recall that $\cA_{S^{[2]}} \subseteq \left\{ xh-y\delta\,|\,x > 0,y > 0 \right\}$.
As explained above, the isometry $\varphi$ induced by $f$ on $\NS(S^{[2]})$ can be of two forms.

\par{{\it First case.}} Assume that 
$\varphi=\left( \begin{matrix}
A & B\\
tB & A
\end{matrix} \right)
$
with $A^2 - tB^2 = 1$ and that $\varphi\neq \id$.
By Proposition~\ref{prop:ample_classes}, the divisors of coordinates $(a, 1)$ with $a \geq 3$ are ample, their images by $\varphi$
have coordinates $(aA + B, atB + A)$ and are ample since $\varphi$ is induced by an automorphism.
This implies immediately $A>0$, $B>0$  and since $\varphi(1,0)=(A,tB)$ is a non-ample class ($h$ is not ample) we get
\begin{equation*}
\cA_{S^{[2]}} \subseteq \left\{ xh-y\delta\,|\,y > 0,Ay < tBx \right\}.
\end{equation*}
The class $\varphi(3, 1)=(3A + B, 3tB + A)$ is ample but it does not satisfy the second inequality, contradiction.

\par{{\it Second case.}} Assume that 
$\varphi=\left( \begin{matrix}
A & B\\
-tB & -A
\end{matrix} \right)
$
with $A^2 - tB^2 = 1$.
Similarly, the classes  $\varphi(a,1)=(aA + B, -atB - A)$ for $a \geq 3$ are ample so $A>0$ and $B<0$ (it is obvious that $B\neq 0$). All the rays
$y=\frac{-atB-A}{aA+B}x$ are contained in $\cA_{S^{[2]}}$, their limit for $a$ big enough is the ray $y=\frac{-tB}{A}x$ so
\begin{equation*}
\cA_{S^{[2]}} \supseteq \left\{ xh-y\delta\,|\,y > 0,Ay < -tBx \right\}.
\end{equation*}
As above the class $\varphi(1,0)=(A,-tB)$ is non-ample  so we get the result.
\end{proof}

\section{The automorphism group of $S^{[2]}$}

As a direct consequence of Proposition~\ref{prop:type_of_auts} we get a first result on the automorphism group of $S^{[2]}$.

\begin{prop}\label{prop: nonsympl inv} Let $S$ be an algebraic K3 surface such 
that $\Pic(S)=\IZ H$ with $H^2=2t$, $t\geq 1$. 
\begin{enumerate}
\item If $t\geq 2$ then the group $\Aut(S^{[2]})$ is either trivial
or isomorphic to $\IZ/2\IZ$, in which case it is generated by a non-symplectic involution.
\item If $t=1$ then $\Aut(S^{[2]})=\{\id_{S^{[2]}},\iota^{[2]}\}\cong \IZ/2\IZ$.
\end{enumerate}
\end{prop}

\begin{proof}
\par If $t\geq 2$, by Lemma~\ref{lem:AutHilbS} the map $\Psi\colon\Aut(S^{[2]})\to O(\NS(S^{[2]}))$ 
is injective. It follows from Proposition~\ref{prop:type_of_auts}  that $\Aut(S^{[2]})$ is either trivial,
or isomorphic to $\IZ/2\IZ$. By \cite[Theorem~4.1]{mongardi} symplectic involutions can exist only when $\rank \NS(S^{[2]}) \geq 8$, so here the non-trivial involution is necessarily non-symplectic.

\par If $t = 1$, by Remark~\ref{rem:t_square} we observe that the only isometry of $O(\NS(S^{[2]}))$ which is induced by 
an automorphism of $S^{[2]}$ is the identity so the map $\Psi$ is trivial.
By Lemmas~\ref{lem:AutS}~\&~\ref{lem:AutHilbS} we get that $\Aut (S^{[2]})$ is isomorphic to $\IZ/2\IZ$, generated by
the involution $\iota^{[2]}$ that is clearly non-symplectic since $\iota$ is non-symplectic.
\end{proof}

From now on we assume that $t\geq 2$ since the case $t=1$ is completely solved by Proposition~\ref{prop: nonsympl inv}.

\subsection{Classes of square two} 

Consider an isometry of $\NS(S^{[2]})$ of the form
\begin{equation}\label{eq: isometries}
\varphi=\left( \begin{matrix}
A & B\\
-tB & -A
\end{matrix} \right), \quad \text{with } A^2 - t B^2 = 1, A > 0, B < 0.
\end{equation}
A direct computation shows that the invariant sublattice of $\NS(S^{[2]})$ for the action of~$\varphi$ is generated by
the vector $(b,a)\coloneqq\frac{1}{d} (-B, A - 1)$ where  $d=\gcd(B,A-1)$ and that $\varphi$ is 
the reflection in the line generated by the vector~$(b,a)$.

If $f\in\Aut(S^{[2]})$ is a non-symplectic involution  the invariant lattice $T(f)$
is a primitive sublattice of $\NS(S^{[2]})$, hence its orthogonal complement $T(f)^\perp$ in $H^2(S,\IZ)$ contains the transcendental
lattice $\Trans(S^{[2]})$. It follows that the isometry $f^\ast$ induced by~$f$ on $H^2(S^{[2]},\IZ)$ is such that
$f^\ast\left|_{\Trans(S^{[2]})}\right. = -\id_{\Trans(S^{[2]})}$. 

\begin{lemma}\label{lem:extending_involutions}
Let $\varphi$ be an involution on $\NS(S^{[2]})$, represented by a matrix of the form~{\rm(\ref{eq: isometries})}.
Then $\varphi$ extends to an involution $\Phi$ on $H^2(S^{[2]}, \IZ)$ such that 
$$
\Phi\left|_{\Trans(S^{[2]})}\right. = -\id_{\Trans(S^{[2]})}
$$ 
if and only if $\varphi$ is the 
reflection through a class of square $2$.
\end{lemma}

\begin{proof}
By \cite[Theorem~1.14.4]{nikulin} the lattice $\Pic(S)=\IZ H$ has a unique primitive embedding in $H^2(S,\IZ)\cong U^{\oplus 3}\oplus E_8^{\oplus 2}$ up to isometry. 
Denoting by $(e,f)$ a basis of the first factor $U$ we can thus assume that it is given by $H\mapsto e+tf$. Since we are working on
$S^{[2]}$ (and not on a deformation of it) we consider the embedding of $\NS(S^{[2]})$ in $H^2(S^{[2]},\IZ)$ given by $h\mapsto e+tf$.

Assume that $\varphi$ extends to an isometry $\Phi$ of $H^2(S^{[2]},\IZ)$ and look at the action on the factor $U \oplus \langle -2\rangle$
with basis $(e+tf,f,-\delta)$. Since the class $w\coloneqq e-tf$ is orthogonal to $\NS(S^{[2]})$ we have $\Phi(w)=-w$.
Writing $w=(e+tf)-2tf$ we get the relation
$$
2t\Phi(f)=(A+1)(e+tf)+tB\delta-2tf.
$$
Recall that $\varphi$ is the reflection in the span of the primitive vector $(b,a)$. An explicit computation gives 
$A=\displaystyle\frac{tb^2+a^2}{tb^2-a^2}$ and $B=\displaystyle\frac{-2ab}{tb^2-a^2}$, hence
$$
\Phi(f)=\frac{b^2}{tb^2-a^2}(e+tf)-\frac{ab}{tb^2-a^2}\delta-f.
$$
It follows that $\frac{b^2}{tb^2-a^2}=k\in\IZ$. We get $(kt-1)b^2=ka^2$ and since $a$ and $b$ are coprime, $b^2$ divides
$k$. This implies
$tb^2-a^2=1$ (since $A>0$) so $(b,a)$ is a class of square $2$.

Conversely, if $(b,a)$ is a class of square $2$ the above computation shows that~$\varphi$ extends
to an isometry $\Phi$ of $U\oplus\langle -2\rangle$. We extend it as $-\id$ to the 
remaining factors $U^{\oplus 2}\oplus E_8^{\oplus 2}$ and we get the result.
\end{proof}

\begin{lemma}\label{lem:2class_exists} Let $S$ be an algebraic K3 surface such 
that $\Pic(S)=\IZ H$, $H^2=2t$, $t\geq 2$. If $f\in\Aut(S^{[2]})$ is not the identity, then its action
on $\NS(S^{[2]})$ is the reflection in the span of a class of square $2$.
\end{lemma}

\begin{proof}
By Lemma~\ref{lem:AutHilbS} and Proposition~\ref{prop:type_of_auts} the isometry induced by $f$ on $\NS(S^{[2]})$ has the form~(\ref{eq: isometries}).
By Proposition~\ref{prop: nonsympl inv} the  involution $f$ is non-symplectic so the invariant lattice $T(f)\subset H^2(S^{[2]},\IZ)$
is a primitive sublattice of $\NS(S^{[2]})$, hence~$T(f)$ is the lattice generated by $(b,a)$.
Moreover by~\cite[Lemma~8.1]{BoissiereCamereSarti} the lattice $T(f)$ is $2$-elementary and 
contains a positive class. It follows that $(b,a)$ has square $2$. 
\end{proof}

\begin{rem}\label{rem:isometry2class}
As a consequence of Lemma~\ref{lem:extending_involutions} and its proof, using the explicit formula for $A$ and $B$ we get ${d=2a}$, hence $A=2a^2+1$ and $B=-2ab$
with $a^2-tb^2=-1$, $a>0$, $b>0$. So the isometry of $\NS(S^{[2]})$ induced by a non-trivial automorphism is
$$
\left( \begin{matrix}
2 a^2 + 1 & -2ab\\
2tab & -2a^2 - 1
\end{matrix} \right)
$$
where $(a,b)$ is a solution of  Pell's equation $P_t(-1)$. This shows that non-trivial automorphisms cannot exist
when $t$ is such that the equation $P_t(-1)$ has no solution. This implies in particular that $t$ is not a square and that the period of the continued fraction expansion of $\sqrt{t}$ is odd.
\end{rem}

\subsection{Main result}\label{ss:main}

The main result of this section is the following theorem which, 
together with Proposition~\ref{prop: nonsympl inv}, gives a complete description of the automorphism
group of $S^{[2]}$ for any value of~$t$.

\begin{theorem}\label{thm: main thm} Let $S$ be an algebraic K3 surface such 
that $\Pic(S)=\IZ H$ with $H^2=2t$, $t\geq 2$. Then $S^{[2]}$ admits a non-trivial automorphism
if and only if one of the following equivalent conditions is satisfied:
\begin{enumerate}
\item\label{item1} $t$ is not a square,  Pell's equation $P_{4t}(5)$ has no solution and
Pell's equation $P_t(-1)$ has a solution.
\item\label{item2} There exists an ample class $D \in \NS(S^{[2]})$ such that $D^2 = 2$.
\end{enumerate}
Moreover, if this is the case the class $D$ is unique, the automorphism is unique and it is a non-symplectic involution.
\end{theorem}

\begin{proof}
If $S^{[2]}$ admits a non-trivial automorphism, by Lemma~\ref{lem:2class_exists} its action $\varphi$ 
on $\NS(S^{[2]})$ is the reflection through a class $(b,a)$ of square $2$. By Remark~\ref{rem:isometry2class} we have that
 $t$ is not a square. By Proposition \ref{prop:type_of_auts} we have that~$\varphi$ is given by the matrix
$$
\left( \begin{matrix}
A & B\\
-tB & -A
\end{matrix} \right) \quad \text{with } A^2 - t B^2 = 1, A > 0, B < 0,
$$
where $A$ and $B$ are determined by the ample cone $\cA_{S^{[2]}}$:
$$
\cA_{S^{[2]}} = \left\{ xh-y\delta\,|\,y  > 0,Ay  < -tBx \right\}.
$$
By Proposition~\ref{prop:nefcone} and Remark~\ref{rem: cone and t},
 Pell's equation $P_{4t}(5)$ has no solution and $(A,-B)$ is the minimal solution of Pell's equation~$P_t(1)$. 
Moreover, by Remark~\ref{rem:isometry2class} $(a,b)$ is a solution of Pell's equation $P_t(-1)$
and we have $A=2a^2+1$, $B=-2ab$. It is easy to check that the class $D=bh-a\delta\in\NS(S^{[2]})$ of square $2$
lives inside the ample cone $\cA_{S^{[2]}}$. This proves (\ref{item1}) and (\ref{item2}).

Assuming (\ref{item1}), by Proposition \ref{prop:nefcone} the ample cone of $S^{[2]}$ is
$$
\cA_{S^{[2]}} = \left\{ xh-y\delta\,|\,y > 0,Ay < -tBx\right\},
$$
where $(A, -B)$ is the minimal solution of $P_t(1)$. Let $(a, b)$ be the minimal solution of $P_t(-1)$.
By Lemma~\ref{lem:Pell} we have $A=2a^2+1$, $B=-2ab$ so again the class $D=bh-a\delta\in\NS(S^{[2]})$ of square $2$
lives inside the ample cone $\cA_{S^{[2]}}$. This proves~(\ref{item2}).

Assuming (\ref{item2}), write $D=bh-a\delta$. By Lemma~\ref{lem:extending_involutions} 
the reflection on $\NS(S^{[2]})$ in the span of $D$ extends to an isometry $\Phi$ of  $H^2(S^{[2]}, \IZ)$ 
such that $\Phi\left|_{\Trans(S^{[2]})}\right. = -\id$ so it induces a Hodge isometry
$\Phi_\IC$ on $H^2(S^{[2]}, \IC)$. Since $\Phi$ leaves invariant the ample class $D$,
this isometry maps the positive cone of $\NS(S^{[2]})$ to itself. By the global Torelli theorem of 
Markman--Verbitsky~\cite[Theorem~1.3, Lemma~9.2]{markman} 
there exists an automorphism $f\in\Aut(S^{[2]})$ such that $f^* = \Phi$, which is a non-symplectic involution
by Proposition~\ref{prop: nonsympl inv} and there is no other non-trivial automorphism on~$S^{[2]}$.

Let us show that the ample class $D$ of square $2$ is unique. Putting $D=bh-a\delta$, we know that $(a,b)$ is
the minimal solution of $P_t(-1)$ and that the ample cone of~$S^{[2]}$ is characterized by $A=2a^2+1$, $B=-2ab$
where $(A,-B)$ is the minimal solution of $P_t(1)$. Assume that $D'=\beta h-\alpha\delta$ is another ample class
of square $2$. Then $\alpha>0$, $\beta>0$ and $(\alpha,\beta)$ is a positive solution of $P_t(-1)$. Putting
$z=a+b\sqrt{t}$ and $w=\alpha+\beta\sqrt{t}$, by Remark~\ref{rem:solPell} we have $w=z^{2n+1}$ with $n\geq 0$.
Assume that $n>0$. Since $z_0=A+B\sqrt{t}=z^2$ we have $w=z^{2n-1}z_0$. Writing $z^{2n-1}=u+v\sqrt{t}$ we get
$$
\alpha=uA-tvB, \quad \beta=vA-uB.
$$
We deduce that $A\alpha +tB\beta = u(A^2-tB^2)=u>0$ so $A\alpha > -tB\beta$, this means that $D'$ is not ample.
So $n=0$ and $D'=D$.
\end{proof}

\section{Examples}\label{s:geometric}

Using the results of Section~\ref{ss:Pell} we find that Pell's equation $P_t(-1)$ has a solution for $t=2,5,10,13,17,\ldots$
and using the software Magma~\cite{Magma}
we find that Pell's equation $P_{4t}(5)$ has no solution for $t=2,10,13,17$, hence by 
Theorem~\ref{thm: main thm} in these cases only~$S^{[2]}$ admits a non-trivial automorphism which is a non-symplectic
involution whose action on $H^2(S,\IZ)$ is the reflection in the span of an ample class~$D$ of square $2$.
By the Hirzebruch--Riemann--Roch theorem we have (see for instance \cite[Section~4]{ogrady2}):
$$
\chi(nD)=\frac{1}{2}n^4+\frac{5}{2}n^2+3.
$$
In particular by Kodaira vanishing theorem $h^0(S^{[2]},D)=6$ so the linear system~$|D|$ defines a rational map 
$$
\varphi_{|D|}\colon S^{[2]}\dasharrow \IP^5.
$$
The non-symplectic involution $\iota$ acts on $H^0(S^{[2]},D)$ and $\phi_{|D|}$ is $\iota$-equivariant.

\subsection{The case $t=2$}

Here $S$ is polarized by an ample class $H$ of square $4$ which is very ample and embeds $S$ as a generic quartic in $\IP^3$. The minimal
solution of Pell's equation $P_2(-1)$ is $(1,1)$ so the non-trivial automorphism acts on $\NS(S^{[2]})$ as the isometry
 given by the matrix
$$
\left(\begin{matrix}
3 & -2\\
4 & -3 
\end{matrix}\right)
$$
which is the reflection in the span of the ample class $D=h-\delta$ of square $2$.
Moreover the ample cone of $S^{[2]}$ is given by $A=3, B=-2$:
$$
\cA_{S^{[2]}} = \left\{ xh-y\delta\,|\,y > 0,3y < 4x\right\}.
$$

By the embedding $S\hookrightarrow \IP (H^0(S,H)^\ast)\cong \IP^3$ we identify the map to
the Grassmannian $\phi\colon S^{[2]}\to \Grass\pa{2, H^0(S, H)^\ast}$ used in Section~\ref{ss:grass}  
 with the map 
$$
\phi\colon S^{[2]}\to \Grass(1,\IP^3),\quad Z\mapsto \langle Z\rangle
$$
that maps $Z\in S^{[2]}$ to the one-dimensional span $\langle Z\rangle$ of $Z$ in $\IP^3$.
Since $S$ has Picard number one it contains no line. Any line in $\IP^3$ intersects $S$ in $4$ points (with multiplicity) so the map $\phi$ is generically $6:1$. We denote by $\psi\colon\Grass(1,\IP^3)\to \IP^5$
the Pl\"ucker embedding, its image $Y$ is the Pl\"ucker quadric. 

By Beauville~\cite{beauville2} the rational map
$\iota\colon S^{[2]}\to S^{[2]}$ that sends $Z$ to the length-two subscheme $Z'$ defined by $\langle Z\rangle\cap S=Z\coprod Z'$
is everywhere defined and we have a commutative diagram
$$
\xymatrix{
S^{[2]}\ar[r]^-\phi\ar[d]_\iota &  \Grass(1,\IP^3)\ar[r]^-\psi& Y\\S^{[2]}\ar[ur]_\phi}
$$
Denoting by $L\coloneqq \psi^\ast \cO_Y(1)$ the very ample line bundle given by the Pl\"ucker embedding 
we get $\iota^\ast\phi^\ast L=\phi^\ast L$ so $\phi^\ast L$ is a multiple of the invariant class $D$.
As explained for instance in \cite[Section~4.1.2]{ogrady1} it is $\phi^\ast L=D$ and $\iota^\ast$ is indeed 
the non-symplectic involution given by the matrix above. The composite map 
$$
f\coloneqq\psi\circ \phi \colon S^{[2]}\to Y\subset \IP^5
$$
is such that $f^\ast \cO_{\IP^5}(1)=D$. Since $H^0(S^{[2]},D)$ has dimension $6$ and since~$Y$ is not contained
in any hyperplane of $\IP^5$ we have $f=\varphi_{|D|}$. In particular
we see in this case that the linear system $|D|$ is base-point-free, the involution $\iota$ acts trivially on
$H^0(S^{[2]},D)$ and $\varphi_{|D|}$ is $\iota$-invariant.

\subsection{The case $t=10$}

Here $S$ is polarized by an ample class of square $20$. The minimal
solution of Pell's equation $P_{10}(-1)$ is $(3,1)$ so the non-trivial automorphism acts on $\NS(S^{[2]})$ as the isometry
 given by the matrix
$$
\left(\begin{matrix}
19 & -6\\
60 & -19 
\end{matrix}\right)
$$
which is the reflection in the span of the ample class $D=h-3\delta$ of square $2$.
Moreover the ample cone of $S^{[2]}$ is given by $A=19, B=-6$:
$$
\cA_{S^{[2]}} = \left\{ xh-y\delta\,|\,y > 0,19y < 60x \right\}.
$$
To our knowledge, there exists no geometric construction of this automorphism in the literature. In particular, it is not known whether the linear system $|D|$ is base-point-free.

\bibliographystyle{amsplain}
\bibliography{Biblio_BCNS_Hilb2K3}

\providecommand{\bysame}{\leavevmode\hbox to3em{\hrulefill}\thinspace}
\providecommand{\MR}{\relax\ifhmode\unskip\space\fi MR }
\providecommand{\MRhref}[2]{%
  \href{http://www.ams.org/mathscinet-getitem?mr=#1}{#2}
}
\providecommand{\href}[2]{#2}
\begin{thebibliography}{10}

\bibitem{bayermacri}
A.~Bayer and E.~Macr\`i, \emph{{MMP} for moduli of sheaves on {$K3$s} via wall
  crossing: nef and movable cones, lagrangian fibrations}, Invent. Math. (to
  appear).

\bibitem{beauville2}
A.~Beauville, \emph{Some remarks on {K}\"ahler manifolds with {$c_{1}=0$}},
  Classification of algebraic and analytic manifolds ({K}atata, 1982), Progr.
  Math., vol.~39, Birkh\"auser Boston, Boston, MA, 1983, pp.~1--26.

\bibitem{beauville1}
\bysame, \emph{Vari\'et\'es {K}\"ahleriennes dont la premi\`ere classe de
  {C}hern est nulle}, J. Differential Geom. \textbf{18} (1983), no.~4, 755--782
  (1984).

\bibitem{BeauvilleInvol}
\bysame, \emph{Antisymplectic involutions of holomorphic symplectic manifolds},
  J. Topol. \textbf{4} (2011), no.~2, 300--304.

\bibitem{belsom}
M.~Beltrametti and A.~J. Sommese, \emph{Zero cycles and {$k$}th order
  embeddings of smooth projective surfaces}, Problems in the theory of surfaces
  and their classification ({C}ortona, 1988), Sympos. Math., XXXII, Academic
  Press, London, 1991, With an appendix by Lothar G{\"o}ttsche, pp.~33--48.

\bibitem{bertramcoskun}
A.~Bertram and I.~Coskun, \emph{The birational geometry of the {H}ilbert scheme
  of points on surfaces}, Birational geometry, rational curves, and arithmetic,
  Springer, New York, 2013, pp.~15--55.

\bibitem{gil}
G.~Bini, \emph{On automorphisms of some {K3} surfaces with {P}icard number
  two}, Annals of the Marie Curie Fellowship Association \textbf{4} (2005),
  1--3.

\bibitem{boissiere}
S.~Boissi{\`e}re, \emph{Automorphismes naturels de l'espace de {D}ouady de
  points sur une surface}, Canad. J. Math. \textbf{64} (2012), no.~1, 3--23.

\bibitem{BoissiereCamereSarti}
S.~Boissi\`ere, C.~Camere, and A.~Sarti, \emph{Classification of automorphisms
  on a deformation family of hyperk\"ahler fourfolds by $p$-elementary
  lattices}, \texttt{arXiv:1402.5154}.

\bibitem{boissieresarti}
S.~Boissi{\`e}re and A.~Sarti, \emph{A note on automorphisms and birational
  transformations of holomorphic symplectic manifolds}, Proc. Amer. Math. Soc.
  \textbf{140} (2012), no.~12, 4053--4062.

\bibitem{Magma}
W.~Bosma, J.~Cannon, and C.~Playoust, \emph{The {M}agma algebra system. {I}.
  {T}he user language}, J. Symbolic Comput. \textbf{24} (1997), no.~3-4,
  235--265, Computational algebra and number theory (London, 1993).

\bibitem{catgo}
F.~Catanese and L.~G{\oe}ttsche, \emph{{$d$}-very-ample line bundles and
  embeddings of {H}ilbert schemes of {$0$}-cycles}, Manuscripta Math.
  \textbf{68} (1990), no.~3, 337--341.

\bibitem{HT}
B.~Hassett and Y.~Tschinkel, \emph{Rational curves on holomorphic symplectic
  fourfolds}, Geom. Funct. Anal. \textbf{11} (2001), no.~6, 1201--1228.

\bibitem{HT2}
\bysame, \emph{Moving and ample cones of holomorphic symplectic fourfolds},
  Geom. Funct. Anal. \textbf{19} (2009), no.~4, 1065--1080.

\bibitem{huybrechts}
H.~Huybrechts, \emph{Lectures on {K3} surfaces},
  \url{http://www.math.uni-bonn.de/people/huybrech/K3Global.pdf}.

\bibitem{andreas}
A.~L. Knutsen, \emph{On {$k$}th-order embeddings of {$K3$} surfaces and
  {E}nriques surfaces}, Manuscripta Math. \textbf{104} (2001), no.~2, 211--237.

\bibitem{andreas2}
\bysame, \emph{Smooth curves on projective {$K3$} surfaces}, Math. Scand.
  \textbf{90} (2002), no.~2, 215--231.

\bibitem{markman}
E.~Markman, \emph{A survey of {T}orelli and monodromy results for
  holomorphic-symplectic varieties}, Complex and differential geometry,
  Springer Proc. Math., vol.~8, Springer, Heidelberg, 2011, pp.~257--322.

\bibitem{Markman2}
\bysame, \emph{Prime exceptional divisors on holomorphic symplectic varieties
  and monodromy reflections}, Kyoto J. Math. \textbf{53} (2013), no.~2,
  345--403.

\bibitem{mongardi}
G.~Mongardi, \emph{Symplectic involutions on deformations of {${\rm
  K}3^{[2]}$}}, Cent. Eur. J. Math. \textbf{10} (2012), no.~4, 1472--1485.

\bibitem{MW}
G.~Mongardi and M.~Wandel, \emph{Induced automorphisms on irreducible
  symplectic manifolds}, 2014, arXiv:1405.5706.

\bibitem{nikulin}
V.~V. Nikulin, \emph{Integer symmetric bilinear forms and some of their
  geometric applications}, Izv. Akad. Nauk SSSR Ser. Mat. \textbf{43} (1979),
  no.~1, 111--177, 238.

\bibitem{ogrady1}
K.~G. O'Grady, \emph{Involutions and linear systems on holomorphic symplectic
  manifolds}, Geom. Funct. Anal. \textbf{15} (2005), no.~6, 1223--1274.

\bibitem{ogrady2}
\bysame, \emph{Irreducible symplectic 4-folds numerically equivalent to
  {$(K3)^{[2]}$}}, Commun. Contemp. Math. \textbf{10} (2008), no.~4, 553--608.

\bibitem{OW}
H.~Ohashi and M.~Wandel, \emph{Non-natural non-symplectic involutions on
  symplectic manifolds of {$K3^{[2]}$}-type}, 2013, arXiv:1305.6353v1.

\bibitem{saintdonat}
B.~Saint-Donat, \emph{Projective models of {$K3$} surfaces}, Amer. J. Math.
  \textbf{96} (1974), 602--639.

\bibitem{sierpinski}
W.~Sierpi{\'n}ski, \emph{Elementary theory of numbers}, second ed.,
  North-Holland Mathematical Library, vol.~31, North-Holland Publishing Co.,
  Amsterdam; PWN---Polish Scientific Publishers, Warsaw, 1988, Edited and with
  a preface by Andrzej Schinzel.

\end{thebibliography}

\end{document}